\documentclass[11pt]{article}

\usepackage{a4wide}
\usepackage{amssymb}
\usepackage{amsmath}
\usepackage{amsthm}
\usepackage{url}

\usepackage{graphicx}
\usepackage{color}
\input{psfig.sty}
\graphicspath{{Figures/}}
\usepackage{wrapfig}

\def\comment#1{}

\def\NI{\noindent}
\def\ni{\noindent}

\def\sk{\smallskip}

\def\term#1{{\bf #1}\marginpar{\raggedright{\small\it #1}}}
\def\nct#1{{\bf #1}} 

\def\ITEMMACRO #1 ??? #2 ???{\par\vskip4pt\noindent%
\hangindent=#2em\setbox0\hbox{#1\kern4pt}%
\ifdim\wd0<\hangindent\setbox0\hbox to\hangindent{\hss#1\kern7pt}\fi%
\box0\ignorespaces}

\def\Item(#1){\ITEMMACRO {\rm (#1)} ??? 1.8 ???}

\def\BrackItem[#1]{\ITEMMACRO [#1] ??? 1.8 ???}

\def\Note{\ITEMMACRO {\NI\bf Note.}  ??? 1.8 ??? \\ \bgroup\sf}
\def\EndNote{\par\egroup\medskip\ni}

\newtheoremstyle{meiner} 
    {4pt}{3pt}           
    {\sffamily\itshape}  
    {}                   
    {\sffamily\bfseries} 
    {}                   
    { }                  
    {}                   
\theoremstyle{meiner}

\newtheorem{theorem}{Theorem}
\newtheorem{lem}{Lemma}

\newtheorem{prop}{Proposition}

\newtheorem{corollary}{Corollary}
\newtheorem{obs}[theorem]{Observation}
\newtheorem{rem}[theorem]{Remark}
\newtheorem{fact}[theorem]{Fact}

\newcommand{\supp}[1]{{\rm{supp}}(#1)}

\newcommand{\sign}[1]{{\rm{sign}}(#1)}
\newcommand{\signvec}[1]{{\overrightarrow{\chi}}(#1)}
\newcommand{\col}[1]{{\rm{col}}(#1)}
\newcommand{\aff}[1]{{\rm{aff}}(#1)}
\newcommand{\spann}[1]{{\rm{span}}(#1)}

\def\Fact#1.{\par\sk{\NI\bf Fact~#1.}\ }
\def\Problem#1.{\par\sk{\NI\bf Problem~#1.}\ }
\def\Claim#1.{\medbreak\ni{\bf Claim~#1.}\ }
\def\Case#1.{\medbreak\ni{\bf Case~#1.}\ }
\def\SubCase#1.{\medbreak\ni{\bf Subcase~#1.}\ }

\def\Remark#1.{\par\sk{\NI\sffamily\bfseries Remark~#1}\ }

\def\Proof{\ni{\sl Proof.}\kern2pt\ }

\definecolor{RED}{rgb}{.84,0,0}
\definecolor{BLUE}{rgb}{0,0,.75}

\begin{document}
%

\title{\vbox{}\vskip-25mm\vbox{}
       {\bf\LARGE Distributive Lattices, Polyhedra,\\
        and Generalized Flow}}

\author{
      \large Stefan Felsner \& Kolja B.~Knauer\\[+1mm]
      \sf Institut f\"ur Mathematik,\\
      \sf Technische Universit\"at Berlin.\\
      \sf {\tt \{felsner,knauer\}@math.tu-berlin.de}
}
\date{}
\maketitle

\begin{abstract}
  A D-polyhedron is a polyhedron $P$ such that if $x,y$ are in $P$ then
  so are their componentwise max and min. In other words, the
  point set of a D-polyhedron forms a distributive lattice with the
  dominance order. We provide a full characterization of the bounding 
  hyperplanes of D-polyhedra.

  Aside from being a nice combination of geometric and order
  theoretic concepts, D-polyhedra are a unifying generalization of
  several distributive lattices which arise from graphs.  In fact every
  D-polyhedron corresponds to a directed graph with arc-parameters,
  such that every point in the polyhedron corresponds to a vertex
  potential on the graph. Alternatively, an edge-based description of
  the point set can be given. The objects in this model are dual to
  generalized flows, i.e., dual to flows with gains and losses.

  These models can be specialized to yield some cases of distributive
  lattices that have been studied previously.  Particular
  specializations are: lattices of flows of planar digraphs (Khuller,
  Naor and Klein), of $\alpha$-orientations of planar graphs
  (Felsner), of c-orientations (Propp) and of $\Delta$-bonds of
  digraphs (Felsner and Knauer). As an additional application we
  exhibit a distributive lattice structure on generalized flow of
  breakeven planar digraphs.
\end{abstract}

\section{Introduction}
\label{sec:intro}
A polyhedron $P\subseteq \mathbb{R}^{n}$ is called
\term{distributive} if 
\begin{center}
$x,y\in P \Longrightarrow
\min(x,y),\max(x,y)\in P$
\end{center}
where minimum and maximum are taken componentwise. Distributive polyhedra are 
abbreviated \term{D-polyhedra}.

   \calc_figscale{45}
    \begin{figure}[htb]
    \centerline{\input{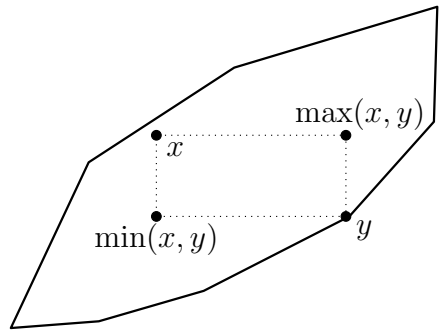}}
    \caption{A distributive polytope in $\mathbb{R}^{2}$.\label{fig:dpolypic}}
    \end{figure}
    

Denote by $\leq_{\text{dom}}$ the \term{dominance order} on $\mathbb{R}^{n}$, i.e.,
$$ 
x\leq_{\text{dom}}y \quad\Longleftrightarrow\quad
 x_{i}\leq y_{i} \quad \text{for all } {1\leq i\leq n}.
$$
The dominance order is a distributive lattice on  $\mathbb{R}^{n}$. Join and meet 
in the lattice are given by the componentwise $\max$ and $\min$. 
\begin{fact}\label{fact:dist}
  A subset $S\subseteq \mathbb{R}^{n}$ is a distributive lattice with
  respect to $x\leq_{\text{dom}}y$ if and only if it is closed with
  respect to $\max$ and $\min$.
\end{fact}
It is a fact from order theory that every \textit{finite} distributive
lattice can be represented as a subset $S\subseteq \mathbb{Z}^{n}$
with the dominance order, see e.g.~\cite{Dav-91}. The name
\textit{distributive} polyhedron is justified by the following:
\begin{obs}
  A polyhedron $P\subseteq \mathbb{R}^{n}$ is a D-polyhedron if and
  only if it is a distributive lattice with respect to the dominance
  order.
\end{obs}

In Section~\ref{sec:geom} we will prove a characterization of
D-polyhedra in terms of their description as an intersection of
halfspaces. In particular we obtain distributivity for known classes
of polytopes,~e.g. \textit{order-polytopes} and more generally
\textit{polytropes}~\cite{Jos-08}, also called \textit{alcoved
  polytopes}~\cite{Lam-06}.

In Section~\ref{sec:comb} we use the geometric characterization of
D-polyhedra to provide a combinatorial description in terms of
vertex-potentials of arc-parameterized digraphs. This is
illustrated by a description of $\Delta$-bonds as integral points of
D-polyhedra in Subsection~\ref{subsec:bonds}. As was shown
in~\cite{Fel-08}, the distributive lattice on $\Delta$-bonds
generalizes distributive lattices on flows of planar
digraphs~\cite{Khu-93}, $\alpha$-orientations of planar
graphs~\cite{Fel-04}, and c-orientations of graphs~\cite{Pro-93}. Here
we additionally suggest to view these objects as integral points in
polyhedra with integral vertices.

In Subsection~\ref{subsec:genp} we give a combinatorial description of the objects in
the arc-space of a parameterized digraph which carry a distributive
lattice structure, coming from a D-polyhedron.

Section~\ref{sec:plan} contains a new application of the theory.
We prove a distributive lattice
structure on a class of generalized flow of planar digraphs.

We conclude in Section~\ref{sec:conco} with final remarks and open problems.

\section{Application}\label{sec:app}

In~\cite{Fel-08} we introduced the set
$\mathcal{B}_{\Delta}(D,c_{\ell},c_{u})$ of (integral) $\Delta$-bonds
of a directed graph. The data is a directed multi-graph $D=(V,A)$ with
upper and lower integral arc-capacities $c_{u},c_{\ell}: A \to
\mathbb{Z}$ and a number $\Delta_C$ for each cycle $C\in\mathcal{C}$. Here a cycle
is understood as a cycle in the underlying undirected graph together
with one of its two cyclic orientations. For a map $x: A \to
\mathbb{Z}$ and $C\in\mathcal{C}$ denote by 
$$
\delta(C,x):=\sum_{a\in
  C^{+}}x(a)-\sum_{a\in C^{-}}x(a)
$$
the \term{circular balance}\footnote{
in previous work on the topic this term was sometimes called the
{\it circular flow difference}. Since bonds are not flows but orthogonal to flows
that name may cause confusion.} 
of $x$ around $C$.
A map $x: A \to \mathbb{Z}$ is a \term{$\Delta$-bond} if

\Item(B$_1$) \quad $c_{\ell}(a)\leq x(a)\leq c_{u}(a)$ for all $a\in A$.
             \hfill (capacity constraints)

\Item(B$_2$) \quad $\Delta_C = \delta(C,x)$ for all $C\in\mathcal{C}$.
             \hfill (circular $\Delta$-balance conditions)

\medskip

\ni 
In~\cite{Fel-08} we showed that $\mathcal{B}_{\Delta}(D,c_{\ell},c_{u})$ 
carries the structure of a distributive lattice.

Below we sketch the crucial observations that lead us to the notion of
D-polyhedra. In this section we will only consider connected
digraphs, i.e., digraphs whose underlying graph is connected. Later we
will see that this can be assumed without loss of
generality. 

The first lemma says that if we allow a change of
arc-capacities we can assume $\Delta=\mathbf{0}$.

\begin{lem}\label{lem:trans}
  For every $D,c_{\ell},c_{u},\Delta$ there are $c'_{\ell},c'_{u}$
  such that
  $\mathcal{B}_{\Delta}(D,c_{\ell},c_{u})\cong\mathcal{B}_{\mathbf{0}}(D,c'_{\ell},c'_{u})$
\end{lem}
\begin{proof}
  Fix a spanning tree $T$ of $D$. Let $f:\mathbb{Z}^{A}\to
  \mathbb{Z}^{A}$ be defined as follows: $f(z)_a:=z_a$ if $a$ is an
  arc of $T$ and $f(z)_a:=z_a-\Delta_{C(T,a)}$ otherwise. Here
  $C(T,a)$ denotes the fundamental cycle of $T$ induced by $a$
  with the cyclic orientation that makes $a$ a 
  forward arc.

  Applying the translation $f$ to $\Delta$-bonds and capacity
  constraints yields a bijection
  $$
  f:\mathcal{B}_{\Delta}(D,c_{\ell},c_{u})\to
        \mathcal{B}_{\mathbf{0}}(D,f(c_{\ell}),f(c_{u})).
  $$
\vskip-8mm
\end{proof}

For elements of $\mathcal{B}_{\mathbf{0}}(D,c_{\ell},c_{u})$ we drop the 
reference to $\Delta = \mathbf{0}$ and simply call them
(integral) \nct{bonds}. 

For a digraph $D$ identify $V_D$ with
$[n]$. Define the \term{network matrix} $N\in\mathbb{R}^{n\times m}$ of $D$ to consist of
columns $e_j-e_i$ for every non-loop arc $a=(i,j)$ and $e_i$ for a
loop $a=(i,i)$. Here $e_{k}$ denotes the $k$th \nct{unit vector},
which has a $1$ in the $k$th entry and is $0$ elsewhere.

\begin{lem}\label{lem:bij-bonds-pot}
  For every tuple $D, c_{\ell},c_{u}$ and $i \in V_D$ there is a bijection between
  $\mathcal{B}_{\mathbf{0}}(D,c_{\ell},c_{u})$ and $P_i:=\{p\in
  \mathbb{Z}^{V}\mid c_{\ell}\leq N^{\top}p\leq c_{u}\text{ and }
  p_i=0\}$, where
  $N$ is the network-matrix of $D$ and $i\in[n]$ is any vertex of $D$.
\end{lem}
\begin{proof}
  We prove that
  $N^T:P_i\to\mathcal{B}_{\mathbf{0}}(D,c_{\ell},c_{u})$, i.e.
  $x(a):=p_k-p_j$ for $a=(j,k)$, is a bijection. Since $D$ is
  connected we can recover any $p\in P_i$ from a given $x$ by taking
  any $(i,j)$-walk $W$ with forward and backward arc set $W^{+}$ and
  $W^{-}$, respectively. Since $\delta(C,x)=0$ for every
  $C\in\mathcal{C}$ mapping $x$ to $p(x)$ with
$$
p_j(x):=\sum_{a\in W^{+}}x(a)-\sum_{a\in W^{-}}x(a)
$$
is independent of the choice of $W$. Injectivity of $N^T$ is a consequence
of the  connectedness of $D$ and from fixing
$p_i=0$. We obtain that $N^T$
is a bijection between $P_i$ and
$\mathcal{B}_{\mathbf{0}}(D,c_{\ell},c_{u})$.
\end{proof}

\begin{lem}\label{lem:distr_P}
The set $P_i$ carries the structure of a distributive lattice.
\end{lem}
\begin{proof}
 Let $p,p'\in P_i$ and let $a=(j,k)$ be and arc of $D$. Then 
\begin{eqnarray*}
	c_{\ell}(a)     &\leq&\min(p_k-p_j,p'_k-p'_j)\\
		&\leq&\min(p_k,p'_k)-\min(p_j,p'_j)\\
		&\leq&\max(p_k-p_j,p'_k-p'_j)\\
		&\leq& c_u(a)
\end{eqnarray*}

\ni
Hence, $\min(p,p') \in P_i \subset\mathbb{Z}^{n}$.  The analog holds for
$\max$, i.e., $P_i$ is closed with respect to componentwise maxima and minima,
which by Fact~\ref{fact:dist} yields that $P_i$ is a distributive sublattice of
the dominance order on $\mathbb{Z}^{n}$.
\end{proof}

We can use the distributive lattice on $P_i$ (Lemma~\ref{lem:distr_P}) and the
bijection (Lemma~\ref{lem:bij-bonds-pot}) to 
induce a distributive lattice on $\mathcal{B}_{\mathbf{0}}(D,c_{\ell},c_{u})$.
Together with Lemma~\ref{lem:trans} this yields a short proof of 

\begin{theorem}\label{thm:bond}
  The set of (integral) $\Delta$-bonds of a connected digraph $D$ within
  capacities $c_{\ell},c_{u}$ carries the structure of a distributive lattice.
\end{theorem}

Since $\mathcal{B}_{\Delta}(D\cup D',c_{\ell}\oplus c'_{\ell},c_{u}\oplus
c'_u)\cong
\mathcal{B}_{\Delta}(D,c_{\ell},c_{u})\times\mathcal{B}_{\Delta}(D',c'_{\ell},c'_{u})$
we can drop the assumption of connectivity of $D$ in the statement of the theorem.

It is shown in~\cite{Fel-08} how to derive
distributive lattices on flows of planar digraphs~\cite{Khu-93},
$\alpha$-orientations of planar graphs~\cite{Fel-04}, and c-orientations of
graphs~\cite{Pro-93} from Theorem~\ref{thm:bond}.

The motivation for the present paper arises from the observation that relaxing
the integrality condition in the above arguments does not destroy the
distributivity. In other words, we obtain a D-polyhedron on the feasible
vertex-potentials. Hence, the set of \textit{real-valued} $\Delta$-bonds inherits a
polyhedral \textit{and} a distributive lattice structure.

In this paper we characterize those real-valued subsets of the arc space of
parameterized digraphs, which can be proven to carry a distributive lattice
structure by the above method as \textit{generalized} $\Delta$-bonds, see
Theorem~\ref{thm:gendelta}.

We will see in Subsection~\ref{subsec:bonds} how Theorem~\ref{thm:bond} turns
out to be a corollary of our theory.
\section{Geometric Characterization}\label{sec:geom}
We want to find a geometric characterization of distributive polyhedra. As a
first ingredient we need the basic
\begin{obs}
The property of being a D-polyhedron is invariant under: 

$\bullet$\quad translation\hbox to 1cm{\hss}
$\bullet$\quad scaling\hbox to 1cm{\hss}
$\bullet$\quad intersection

\end{obs}

\ni In order to give a description of D-polyhedra in terms of bounding
halfspaces we will pursue the following strategy. We start by characterizing
distributive affine subspaces of~$\mathbb{R}^{n}$. Then we provide a
characterization of the orthogonal complements of distributive affine spaces.
Finally we show that D-polyhedra are exactly those polyhedra that have a
representation as the intersection of distributive halfspaces.

\subsection{Distributive Affine Space}

For a vector $x\in\mathbb{R}^{n}$ let $\supp{x}:=\{i\in [d]\mid x_{i}\neq 0\}$
be its \nct{support}. Moreover denote by $x^+:=\max(\mathbf{0},x)$ and
$x^-:=\min(\mathbf{0},x)$. Call a set of vectors $B\subseteq \mathbb{R}^{n}$
\term{NND (non-negative disjoint)} if the elements of $B$ are non-negative and
have pairwise disjoint supports. Note that a NND set of non-zero vectors is
linearly independent.

\begin{lem}\label{lem:bas}
  Let $I\cup\{x\}\subset \mathbb{R}^n$ be linearly independent, then
  $I\cup\{x^+\}$ or $I\cup\{x^-\}$ is linearly independent.
\end{lem}
\begin{proof}
  Suppose there are linear combinations $x^+ = \sum_{b\in I}\mu_b b$
  and $x^- = \sum_{b\in I}\nu_b b$. Hence $x = \sum_{b\in I}(\mu_b+\nu_b)x_b$,
  which proves that $I\cup\{x\}$ is linearly dependent --
  contradiction.
\end{proof}

\begin{prop}\label{pro:bas}
  An affine subspace $A\subseteq \mathbb{R}^{n}$ is a D-polyhedron if and only
  if it has a non-negative disjoint basis $B$.
\end{prop}
\begin{proof} Since the properties involved are invariant under translation,
  we show the statement for linear subspaces only.

  ``$\Longleftarrow$'': Let $x,y\in A$ and $x=\sum_{b\in B}\mu_{b}b$ and
  $y=\sum_{b\in B}\nu_{b}b$ their representations with respect to a NND basis
  $B$ of $A$. Since the supports of vectors in $B$ are disjoint $x_{i}< y_{i}$
  is equivalent to $x_{i}=\mu_b b_{i}< \nu_b b_{i}=y_{i}$ for the unique $b\in
  B$ with $i\in \supp{b}$. Since every $b\in B$ is non-negative this is
  equivalent to $\mu_b < \nu_b$. This implies
  $\max(x_{i},y_{i})=\max(\mu_b,\nu_b)b_{i}$. In vectors this reads:
$$
\max(x,y)=\max(\sum_{b\in B}\mu_{b}b,\sum_{b\in B}\nu_{b}b)=\sum_{b\in B}\max(\mu_{b},\nu_{b})b.
$$
The analog holds for minima, hence $x,y\in
A\Longrightarrow\max(x,y),\min(x,y)\in A$, i.e., $A$ is distributive.

``$\Longrightarrow$'': Let $A$ be distributive and $I\subset A$ a NND set of
support-minimal non-zero vectors. If $I$ is not a basis of $A$ there is $x\in
A$ with: {\Item(1) $I\cup\{x\}$ is linearly independent.
\Item(2) $\exists_{i\in [n]}:x_i>0$.
\Item(3) $\supp{x}$ is minimal among the vectors with (1) and (2).}

\bigskip

\ni Claim: $I\cup\{x\}$ is NND. 

If $x$ is not non-negative then $x^+$ and $-x^-$ are non-negative and have
smaller support than~$x$. By Lemma~\ref{lem:bas} one of $I\cup\{x^+\}$ and
$I\cup\{-x^-\}$ is linearly independent -- a contradiction to the
support-minimality of $x$.

If there is $b\in I$ such that $\supp{x}\cap\supp{b}\neq\emptyset$ choose a
maximal $\mu\in\mathbb{R}$ such that for some coordinate $j$ we have
$x_j=\mu b_j$. We distinguish two cases.

If $\supp{x}\subseteq\supp{b}$ then $\emptyset\neq\supp{\mu b-x}\subsetneq
\supp{b}$ contradicts the support-minimality in the choice of $b\in I$.

If $\supp{x}\nsubseteq\supp{b}$ then since $I\cup\{\mu b-x\}$ is linearly
independent one of $I\cup\{(\mu b-x)^+\}$ and $I\cup\{(x-\mu b)^-\}$ is
linearly independent by Lemma~\ref{lem:bas}. This contradicts
support-minimality in the choice of $x$.
\end{proof}

\begin{prop}
 An NND basis is unique up to scaling.
\end{prop}
\begin{proof}
  Suppose $A\subseteq \mathbb{R}^n$ has NND bases $B$ and $B'$. Suppose there
  are $b\in B$ and $b'\in B'$ such that
  $\emptyset\neq\supp{b}\cap\supp{b'}\neq\supp{b'},\supp{b}$. By
  Proposition~\ref{pro:bas} we have $\min(b,b')\in A$ but $\supp{\min(b,b')}$
  is strictly contained in the supports of $b$ and $b'$. Since $B$ and $B'$
  are NND the vector $\min(b,b')$ can neither be linearly combined by $B$ nor
  by $B'$.

  So the supports of vectors in $B$ and $B'$ induce the same partition of
  $[n]$. Since they are NND, the vectors $b\in B$ and $b'\in B'$ with
  $\supp{b}=\supp{b'}$ must be scalar multiples of each other.
\end{proof}

The next step is to define a class of network matrices
of arc-parameterized digraphs such that an
affine space $A$ is distributive if and only if there is a 
network matrix $N_{\Lambda}$ in the class such that
$A=\{p\in \mathbb{R}^{n}\mid N_{\Lambda}^{\top}p=c\}$.

We call a tuple $D_{\Lambda}=(V,A,\Lambda)$ an \term{arc-parameterized
  digraph} if $D=(V,A)$ is a directed multi-graph -- the \nct{underlying
  digraph} -- with $V=[n]$, $|A|=m$, and $\Lambda\in\mathbb{R}^{m}_{\geq
  \mathbf{0}}$ has the property that $\lambda_{a}=0$ only if $a$ is a loop.
For emphasis we repeat: All arc-weights $\lambda_a$ are non-negative.

Given an arc parameterized digraph $D_{\Lambda}$ we define its
\term{generalized network-matrix} to be the matrix
$N_{\Lambda}\in\mathbb{R}^{n\times m}$ with a column
$e_{j}-\lambda_{a}e_{i}$ for every arc $a=(i,j)$ with parameter
$\lambda_{a}$.

\begin{prop}\label{pro:ort}
  Let $A\subseteq \mathbb{R}^{n}$ be a non-empty affine subspace. 
  Then $A$ is distributive if and only if
  $A=\{p\in\mathbb{R}^{n}\mid N^{\top}_{\Lambda}p=c\}$, where
  $N_{\Lambda}$ is some generalized network-matrix. Moreover, $N_{\Lambda}$
  can be chosen such that the corresponding 
  arc parameterized digraph $D_{\Lambda}$ is a disjoint union of trees and loops.
\end{prop}
\begin{proof}
  Since the properties involved are invariant under translation, we can assume
  $A$ to be linear, hence $c=\mathbf{0}$.

  ``$\Longrightarrow$'': By Proposition~\ref{pro:bas} the distributive $A$ has
  a NND basis $B$. We construct an arc-parameterized digraph $D_{\Lambda}$,
  such that the columns of its generalized network-matrix $N_{\Lambda}$ form a
  basis of the orthogonal complement of $A$.

  For every $b\in B$ choose some arbitrary directed spanning tree on
  $\supp{b}$. To an arc $a=(i,j)$ with $i,j\in\supp{b}$ we associate the arc
  parameter $\lambda_{a}:=b_{j}/b_{i}>0$. For every $i\notin\bigcup_{b\in
    B}\supp{b}$ insert a loop $a=(i,i)$ with $\lambda_{a}:=0$. Collect the
  $\lambda_{a}$ of all the arcs in a vector
  $\Lambda\in\mathbb{R}^{m}_{\geq\mathbf{0}}$. The resulting arc-parameterized
  digraph $D_{\Lambda}$ is a disjoint union of loops and directed trees.

  Denote by $\col{N_{\Lambda}}$ the set of columns of $N_{\Lambda}$. If $b\in
  B$ and $z_a\in\col{N_{\Lambda}}$, then either
  $\supp{b}\cap\supp{z_a}=\emptyset$ -- this holds for all $b$ when $a=(i,i)$ is  a loop and $z_a=e_i$ -- or $\langle b,z_a\rangle=b_j-\lambda_a
  b_i=b_j-(b_{j}/b_{i})b_i=0$ for $a=(i,j)$. Therefore, $\col{N_{\Lambda}}$ is
  orthogonal to~$A$. Since the underlying digraph of $D_{\Lambda}$ consists of
  trees and loops only, $\col{N_{\Lambda}}$ is linearly independent. To
  conclude that $\col{N_{\Lambda}}$ generates $A^{\bot}$ in $\mathbb{R}^n$ we
  calculate
$$
|B|+|\col{N_{\lambda}}|=|B|+(\sum_{b\in
  B}(|\supp{b}|-1)+|[n]\backslash\bigcup_{b\in B}\supp{b}|=\sum_{b\in
  B}|\supp{b}|+n-|\bigcup_{b\in B}\supp{b}|.
$$
Since the supports in $B$ are mutually disjoint this equals $n$.

``$\Longleftarrow$'': Let $D_{\Lambda}$ be an arc parameterized
digraph such that $N^{\top}_{\Lambda}p=\mathbf{0}$ has a solution. If
$a=(i,j)$ is an arc, then $p_i-\lambda_{a}p_j = 0$, hence to know $p$
it is enough to know $p_i$ for one vertex~$i$ in each connected
component of $D_{\Lambda}$.  Therefore, the affine space of solutions
is unaffected by deleting an edge from a cycle of $D_{\Lambda}$. This
shows that there is no loss of generality in the assumption that the
underlying digraph $D$ of $D_{\Lambda}$ is a disjoint union of trees
and loops.  Under this assumption we construct a NND basis of
$\{p\in\mathbb{R}^{n}\mid N^{\top}_{\Lambda}p=\mathbf{0}\}$: For every
tree-component~$T$ of $D$ define a vector $b$ with $\supp{b}=V(T)$ as
follows: choose some $i\in V(T)$ and set $b_{i}:=1$, for any other
$j\in V(T)$ consider the $(i,j)$-walk $W$ in $T$.  Define
$b_{j}:=\Pi_{a\in W^{+}}\lambda_{a}\Pi_{a\in W^{-}}\lambda_{a}^{-1}$
where $W^{+}$ and $W^{-}$ are the sets of forward and backward arcs on
$W$. Since arc-weights are non-negative this procedure yields an NND
set $B$ set of non-zero vectors. Note that $B$ is orthogonal to
$\col{N_{\Lambda}}$ and that $\col{N_{\Lambda}}$ is a linearly
independent set with as many vectors as there are arcs in $A(D)$.
Denote by $k$ the number of tree-components of $D$. To see that $B$ is
spanning, we calculate $$|B|+|\col{N_{\Lambda}}|=k+|A(D)|=k+n-k=n.$$
Hence, $\text{span}(B) = \{p\in\mathbb{R}^{n}\mid N^{\top}_{\Lambda}p=\mathbf{0}\}$.
\end{proof}

\subsection{Distributive Polyhedra}

For a polyhedron $P$ we define $F\subseteq P$ to be a \nct{face} if there is
$A=\{p\in\mathbb{R}^{n}\mid \langle z,p\rangle= c\}$ such that $F=P\cap A$ and
$P$ is contained in the \nct{induced} \nct{halfspace}
$A^{+}:=\{p\in\mathbb{R}^{n}\mid \langle z,p\rangle\leq c\}$.

\begin{lem}\label{lem:fac}
 Faces of D-polyhedra are D-polyhedra.
\end{lem}
\begin{proof}
  Let $P$ be a D-polyhedron such that $P\subseteq A^{+}=\{p\in\mathbb{R}^{n}\mid
  \langle z,p\rangle\leq c\}$ and let $F=P\cap A$ a face. Suppose there are
  $x,y\in F$ such that $\max(x,y)\not\in F$, i.e., $\langle z,\max(x,y)\rangle<
  c$. Since $2c=\langle z,x+y\rangle=\langle z,\max(x,y)\rangle+\langle
  z,\min(x,y)\rangle$ this implies $\min(x,y)\not\in P$ -- contradiction.
\end{proof}

\begin{lem}\label{lem:aff}
 The affine hull of a D-polyhedron is distributive.
\end{lem}
\begin{proof}
  Let $P$ be a $D$-polyhedron and $x,y\in\aff{P}$. Scale $P$ to $P'$ such that
  $x,y\in P'\subseteq \aff{P}$. Since scaling preserves distributivity
  $\min(x,y),\max(x,y)\in P'\subseteq \aff{P}$.
\end{proof}

\begin{lem}\label{lem:hal}
  Let $z\in\mathbb{R}^{n}$ and $c\in\mathbb{R}$. If
  $A=\{p\in\mathbb{R}^{n}\mid \langle z,p\rangle=c\}$ is distributive then the
  \term{halfspace} $A^{+}=\{p\in\mathbb{R}^{n}\mid \langle z,p\rangle\leq c\}$
  is distributive as well.
\end{lem}
\begin{proof}
  Suppose that $x,y\in A^{+}$ such that $\max(x,y)\notin
  A^{+}$. The line segments $[x,\max(x,y)]$ and $[y,\max(x,y)]$ 
  contain points $x',y'\in A$ such that
  $\max(x',y')=\max(x,y)$. This contradicts the distributivity of $A$.
\end{proof}

\begin{theorem}\label{thm:geom}
  A polyhedron $P\subseteq \mathbb{R}^{n}$ is a D-polyhedron if and
  only if
  $$
  P=\{p\in\mathbb{R}^{n}\mid N_{\Lambda}^{\top}p\leq c\}
  $$
  for some generalized network-matrix $N_{\Lambda}$ and $c\in\mathbb{R}^{m}$.
\end{theorem}
\begin{proof}
  ``$\Longrightarrow$'': By Lemma~\ref{lem:fac} every face $F$ of $P$ is
  distributive. Lemma~\ref{lem:aff} ensures that $\aff{F}$ is distributive.
  Proposition~\ref{pro:ort} yields $\aff{F}=\{p\in\mathbb{R}^{n}\mid
  N(F)_{\Lambda(F)}^{\top}p=c(F)\}$ for a generalized network-matrix
  $N(F)_{\Lambda(F)}$. In particular this holds for for $\aff{P}$. Now if $F$
  is a facet of $P$ every row $z$ of $N(F)_{\Lambda(F)}^{\top}$ is a
  generalized network-matrix as well. Choose a row $z_F$ such that
  $A_F^+:=\{p\in\mathbb{R}^{n}\mid \langle z_F,p\rangle\leq c_F\}$ is a facet-defining 
  halfspace for $F$.

  By the Representation Theorem for Polyhedra~\cite{Zie-95} we can write
$$
P=(\bigcap_{F\text{ facet}}A_F^+)\cap\aff{P}.
$$
The above chain of arguments yields
$$
P=(\bigcap_{F\text{ facet}}\{p\in\mathbb{R}^{n}\mid \langle z_F,p\rangle\leq
c_F\})\cap\{p\in\mathbb{R}^{n}\mid N(P)_{\Lambda(P)}^{\top}p=c(P)\}.
$$
Here the single matrices involved are generalized network-matrices. Glueing
all these matrices horizontically together one obtains a single generalized
network-matrix $N_{\Lambda}$ and a vector $c$ such that
$P=\{p\in\mathbb{R}^{n}\mid N_{\Lambda}^{\top}p\leqq c\}$. It remains to show,
that we can replace defining equalities by inequalities, while preserving a
network-matrix representation. We distinguish two cases.
{
\Item(1) If $\lambda_{a}\neq 0$ we have $\langle
e_j-\lambda_{a}e_i,p\rangle =c_a \Leftrightarrow (\langle e_j-\lambda_{a}e_i
,p\rangle\leq c_a\text{ and }\langle e_i-\lambda_{a}^{-1}e_j,p\rangle\leq
-\lambda_{a}^{-1}c_a)$.
\Item(2) If $\lambda_{a}=0$ since $a=(i,i)$ must be a loop of $D_{\Lambda}$ we 
have $\langle e_i-0e_i,p\rangle=c_a$, which can be replaced by 
$(\langle e_i-0e_i,p\rangle\leq c_a\text{ and } \langle e_i-2e_i,p\rangle\leq
-c_a)$.
\par}
\medskip

\ni
In each of the cases a single arc with an equality-constraint is
replaced by a pair of oppositely oriented arcs. 
This shows that we can write $P$ as $P=\{p\in\mathbb{R}^{n}\mid
N_{\Lambda}^{\top}p\leq c\}$, for some generalized
network-matrix $N_{\Lambda}$.

``$\Longleftarrow$'': If $P=\{p\in\mathbb{R}^{n}\mid
N_{\Lambda}^{\top}p\leq c\}$ then $P$ is the intersection of bounded
halfspaces, which are distributive by Lemma~\ref{lem:hal}, because
their defining hyperspaces are distributive by
Proposition~\ref{pro:ort}. Since intersection preserves
distributivity, $P$ is a D-polyhedron.
\end{proof}

\begin{rem}\label{rem:equal}
From the proof it follows that the system $N_{\Lambda}^{\top}p\leqq c$
with equality- and inequality-constraints
defines a D-polyhedron whenever $N_{\Lambda}$ is a generalized network-matrix.
\end{rem}

As an immediate application of the theorem we obtain that \textit{order polytopes}
($\Lambda,c\in\{0,1\}^{m}$) are D-polytopes. More generally
($\Lambda\in\{0,1\}^{m}$ and $c\in\mathbb{Z}^{m}$) one obtains distributivity
for a more general class of polytopes that has been named
\textit{polytropes} in~\cite{Jos-08} and  \textit{alcoved
  polytopes} in~\cite{Lam-06}. We will return to this class in
Subsection~\ref{subsec:bonds}.

\begin{rem}\label{rem:scale}
  Generalized network matrices are not the only matrices that can be used to
  represent D-polyhedra. 
\end{rem}

To see this observe that scaling columns of $N_{\lambda}$ and entries
of $c$ simultaneously preserves the polyhedron but may destroy the
property of the matrix. There may, however, be representations of 
different type. Consider e.g., the D-polyhedron
consisting of all scalar multiples of $(1,1,1,1)$
in $\mathbb{R}^{4}$, it can be described by the   
six inequalities
 $\sum_{i\in A} x_i - \sum_{i\not\in {A}} x_i \leq 0$, for $A$
a 2-subset of $\{1,2,3,4\}$.

\section{Combinatorial Model}\label{sec:comb}

We have shown that a D-polyhedron $P$ is completely described by an
arc-parameterized digraph $D_{\Lambda}$ and an arc-capacity vector
$c\in\mathbb{R}^{m}$. This characterization suggests to consider the points
of $P$ as `graph objects'.  A \term{potential} for $D_{\Lambda}$ is a
vector $p\in\mathbb{R}^{n}$, which assigns a real number $p_{i}$ to
each vertex $i$ of $D_{\Lambda}$, such that the inequality
$p_{j}-\lambda_{a}p_{i}\leq c_{a}$ holds for every arc $a=(i,j)$ of
$D_{\Lambda}$.  The points of the D-polyhedron $P(D_{\Lambda})_{\leq
  c}=\{p\in\mathbb{R}^{n}\mid N_{\Lambda}^{\top}p\leq c\}$ are exactly
the potentials of $D_{\Lambda}$.
Theorem~\ref{thm:geom} then can be rewritten
\begin{theorem}
  A polyhedron is distributive if and only if it is the set of 
  potentials of an arc-parameterized digraph $D_{\Lambda}$.
\end{theorem}

Interestingly there is a second class of graph objects associated with
the points of a \hbox{D-polyhedron}. While potentials are weights on
vertices this second class consists of weights on the arcs of
$D_{\Lambda}$.  Given a D-polyhedron $P=\{p\in\mathbb{R}^{n}\mid
N_{\Lambda}^{\top}p\leq c\}$ we look at
$\mathcal{B}(D_{\Lambda})_{\leq c}:=\{x\in\mathbb{R}^{m}\mid {x \leq
  c} \text{ and }x\in\text{Im}(N_{\Lambda}^{\top})\}$. Note that $p\in
P$ if and only if $ N_{\Lambda}^{\top}p \in
\mathcal{B}(D_{\Lambda})_{\leq c}$, i.e.,
$\mathcal{B}(D_{\Lambda})_{\leq c} = N_{\Lambda}^{\top}P$.

In the spirit of the terminology of \textit{generalized flow},
c.f.~\cite{Ahu-93}, the elements of $\mathcal{B}(D_{\Lambda})_{\leq c}$ will
be called \term{generalized bonds}.

\begin{theorem}\label{thm:genbond}
  Let $D_{\Lambda}$ be an arc-parameterized digraph with capacities $c\in
  \mathbb{R}^{m}$. The set $\mathcal{B}(D_{\Lambda})_{\leq c}$ inherits the
  structure of a distributive lattice from a bijection $P'\to
  \mathcal{B}(D_{\Lambda})_{\leq c}$, where $P'$ is a  D-polyhedron
  that can be obtained from $P = P(D_{\Lambda})_{\leq c}$ by intersecting
  $P$ with some hyperplanes of type $H_i =\{ x \mid x_i=0\}$.
\end{theorem}

\begin{proof}
  Since $\mathcal{B}(D_{\Lambda})_{\leq c}=N^{\top}_{\Lambda}P$ for the
  D-polyhedron of feasible vertex-potentials of $D_{\Lambda}$, and
  $N_{\Lambda}^{\top}$ is a linear map, the set of generalized bonds is a
  polyhedron.

 If $N_{\Lambda}^{\top}$ is bijective on $P$ the set
 $\mathcal{B}(D_{\Lambda})_{\leq c}$ inherits the distributive lattice
 structure from $P$. This is not always the case. Later we show that we can
 always find a D-polyhedron $P'\subseteq P$ such that $N_{\Lambda}^{\top}$ is
 a bijection from $P'$ to $\mathcal{B}(D_{\Lambda})_{\leq c}$.

 From Proposition~\ref{pro:ort} we know that
 $\text{Ker}(N_{\Lambda}^{\top})$ is a distributive space. 
 By Proposition~\ref{pro:bas} there is a NND basis $B$ of
 $\text{Ker}(N_{\Lambda}^{\top})$. 
 For every $b\in B$ fix an arbitrary element $i(b)\in\supp{b}$. Denote the set of
 these elements by $I(B)$. Define $A:=\spann{\{e_i\in\mathbb{R}^n\mid
   i\in[n]\backslash I(B)\}}$.

 {
\Item(1) $A$ is distributive:
\\
 By definition $A$ has a NND basis, i.e. is distributive by
 Proposition~\ref{pro:bas}.

 \Item(2) $\mathcal{B}(D_{\Lambda})=N^{\top}_{\Lambda}A$:
\\
 Let $N_{\Lambda}^{\top}p=x\in \mathcal{B}(D_{\Lambda})$. Define
 $p':=p-\sum_{b\in B}(p_{i(b)}b)$. Since $\sum_{b\in B}(p_{i(b)}b)\in
 \text{Ker}(N_{\Lambda}^{\top})$ we have $N_{\Lambda}^{\top}p'=x$. Moreover
 $p'_i=0$ for all $i\in I(B)$, i.e. $p'\in A$.

 \Item(3) $N^{\top}_{\Lambda}:A\hookrightarrow \mathcal{B}(D_{\Lambda})$ is
 injective:
\\
 Suppose there are $p,p'\in A$ such that
 $N_{\Lambda}^{\top}p=N_{\Lambda}^{\top}p'$. Then $p-p'\in
 \text{Ker}(N_{\Lambda}^{\top})\cap A$. But by the definition of $A$ this
 intersection is trivial, i.e., $p=p'$.
\par}
\medskip

\ni
We have shown that $A\cong \mathcal{B}(D_{\Lambda})$ via $N^{\top}_{\Lambda}$
and that $A$ is distributive. Thus $P':=P\cap A$ is a D-polyhedron such that
the map of the matrix $N^{\top}_{\Lambda}$ is a  bijection from 
$P'$ to $\mathcal{B}(D_{\Lambda})_{\leq c}$.
\end{proof}

The intersection of $P$ with $H_i$ can be modelled by adding a loop
$a=(i,i)$ with capacity $c_a=0$ to the digraph. Hence, with Remark~\ref{rem:equal}
the preceding theorem says that for every $\mathcal{B}(D_{\Lambda})_{\leq c}$
we can add some loops to yield a graph $D'_{\Lambda'}$ and capacities $c'$ such that
$$
\mathcal{B}(D_{\Lambda})_{\leq c}=\mathcal{B}(D'_{\Lambda'})_{\leq c'}\cong
P(D'_{\Lambda'})_{\leq c'}=P'.
$$ 
In the following we will always assume to be given generalized bonds
$\mathcal{B}(D_{\Lambda})_{\leq c}$ such that
$\mathcal{B}(D_{\Lambda})_{\leq c} \cong P(D_{\Lambda})_{\leq c}$.
In this case we call $(D_{\Lambda},c)$ \term{reduced}.

Note that $\mathcal{B}(D_{\Lambda})_{\leq c}$ can be far from being a
D-polyhedron, but it inherits the distributive lattice structure via an
isomorphism from a D-polyhedron.

In the following we investigate generalized bonds, i.e., the elements
of $\mathcal{B}_{\Lambda}(D)$, as objects in their own right.  Since
$\mathcal{B}_{\Lambda}(D)=\text{Im}(N_{\Lambda}^{\top}) =
\text{Ker}(N_{\Lambda})^{\bot}$ we have $\langle x,f\rangle=0$ for
all $x\in \mathcal{B}_{\Lambda}(D)$ and $f\in
\text{Ker}(N_{\Lambda})$. Understanding the elements of
$\text{Ker}(N_{\Lambda})$ as objects in the arc space of
$D_{\Lambda}$ will be vital to our analysis. In
Subsection~\ref{subsec:bonds} we will review the case of ordinary
bonds, which leads to a description closely related to the definition of
$\Delta$-bonds in Section~\ref{sec:app}. Recall that this definition
was based on the notion of \textit{circular balance}. In
Subsection~\ref{subsec:genp} we will then be able to describe the
generalized bonds of $D_{\Lambda}$ as capacity-respecting arc values,
which satisfy a \textit{generalized circular balance condition} around
elements of $\text{Ker}(N_{\Lambda})$, see
Theorem~\ref{thm:genflowdiff}.

\subsection{Bonds}\label{subsec:bonds}

Consider as an example the case where $D$ is a digraph without loops
and $\Lambda=\mathbf{1}$. In this case~$N_{\Lambda}$ is nothing but the
network-matrix $N$ of $D$. The elements of
$\text{Ker}(N)=:\mathcal{F}(D)$ are the flows of $D$, i.e those real
arc values $f\in \mathbb{R}^{m}$ which respect flow-conservation at
every vertex of $D$. Moreover, each support-minimal element of
$\mathcal{F}(D)$ is a scalar multiple of the \term{signed incidence
  vector} $\signvec{C}$ of a cycle $C$ of $D$, where $\signvec{C}_{a}$ is
$+1$ if $a$ is a forward arc of $C$, and~$-1$ if $a$ is a backward
arc, and $0$ otherwise.  The set $\mathcal{B}(D)$ of generalized bonds
of $D$ consists of those $x\in \mathbb{R}^{m}$ with $\langle
x,f\rangle=0$ for all flows $f$. This is equivalent to $\langle
x,\signvec{C}\rangle=0$ for all $C\in\mathcal{C}$. Now $\langle
  x,\signvec{C}\rangle =\delta(C,x)$ (see Section~\ref{sec:app}), hence
  $\mathcal{B}(D)_{\leq c}$ can be viewed as the set of
  \textit{real-valued} $\mathbf{0}$-bonds of $D$ respecting the arc
  capacities $c$.

Theorem~\ref{thm:genbond} yields a distributive lattice structure on
the set of real-valued $\mathbf{0}$-bonds of an arbitrary digraph $D$. 
We may use Lemma~\ref{lem:trans} to conclude
Theorem~\ref{thm:bond} if we can prove distributivity on the
\textit{integral} bonds. To this end we first we make the following
\begin{obs}\label{obs:inter}
  The intersection of a D-polytope $P\subseteq \mathbb{R}^{n}$ and any other
  (particularly finite) distributive sublattice $L$ of $\mathbb{R}^{n}$ yields
  a distributive lattice $P\cap L$.
\end{obs}
So if $P\subseteq \mathbb{R}^{n}$ is a D-polyhedron then
$P\cap\mathbb{Z}^{n}$ is a distributive lattice. Since by
Theorem~\ref{thm:genbond} we can assume $N^{\top}$ to be bijective on $P$
we obtain a distributive lattice structure on  $N^{\top}(P\cap \mathbb{Z}^{n})$.

However, we want a distributive lattice on integral bonds, i.e., on
$\mathcal{B}(D)_{\leq c}\cap\mathbb{Z}^{m}$. Luckily~$N$ is a
totally unimodular matrix, which yields $\mathcal{B}(D)_{\leq
  c}\cap\mathbb{Z}^{m}=N^{\top}(P\cap \mathbb{Z}^{n})$, see~\cite{Sch-86},
i.e. the integral bonds carry a distributive lattice structure.

\subsection{General Parameters}\label{subsec:genp}

Lets now look at the case of general bonds of an arc-parameterized digraph
$D_{\Lambda}$. The aim of this section is to describe
$\mathcal{B}(D_{\Lambda})_{\leq c}$ as the orthogonal complement of
$\text{Ker}(N_{\Lambda})$ within the capacity bounds given by $c$. For
$f\in\mathbb{R}^{m}$ and $j\in V$ we define the \term{excess} of $f$ at $j$ as
$$
\omega(j,f):=(\sum_{a=(i,j)}f_{a})-(\sum_{a=(j,k)}\lambda_{a}f_{a}).
$$
Since $f\in\text{Ker}(N_{\Lambda})$ means $\omega(v,f)=0$ for all
$v\in V$ we think of $f$ as an edge-valuation satisfying a
\textit{generalized flow-conservation}. This justifies the name
\term{generalized flow} for elements of $\text{Ker}(N_{\Lambda})$. 
Generalized flows were introduced by Dantzig~\cite{Dan-63} in the sixties 
and there has been much interest in related algorithmic problems. For surveys on the 
work, see~\cite{Ahu-93, Tru-77}. The most efficient
algorithms known today have been proposed in~\cite{Fle-02}.

We will denote $\mathcal{F}(D_{\Lambda}):=\text{Ker}(N_{\Lambda})$ and
call it the \nct{generalized flow space}. Let ${\mathcal{C}}(D_{\Lambda})$ 
be the set of support-minimal vectors of
$\mathcal{F}(D_{\Lambda})\backslash\{\mathbf{0}\}$, i.e., 
$f\in{\mathcal{C}}(D_{\Lambda})$ if and only if $\supp{g}\subseteq\supp{f}$ implies 
$\supp{g}=\supp{f}$ for all $g\in\mathcal{F}(D_{\Lambda})\backslash\{\mathbf{0}\}$. 
Elements of ${\mathcal{C}}(D_{\Lambda})$ will be called 
\term{generalized cycles}. Since the support-minimal
vectors ${\mathcal{C}}(D_{\Lambda})$ span the entire space
$\mathcal{F}(D_{\Lambda})$; the generalized bonds of $D_{\Lambda}$ are
already determined by being orthogonal to
${\mathcal{C}}(D_{\Lambda})$, i.e., to all generalized cycles.

\smallskip

\centerline{\textit{What do generalized cycles look like?}}
\smallskip

For a
loop-free oriented arc set $S$ of $D_{\Lambda}$ define its \term{multiplier}
as 
$$
\lambda(S):=\prod_{a\in S}\lambda_{a}^{\signvec{S}_{a}},
$$
where $\signvec{S}_a=\pm 1$ depending on the orientation of $a$ in $S$.
 
A cycle $C$ in the underlying graph with a cyclic orientation will be
called \term{lossy} if $\lambda(C)<1$, and \term{gainy} if
$\lambda(C)>1$, and \term{breakeven} if $\lambda(C)=1$. A
\term{bicycle} is an oriented arc set that can be written as 
$C\cup W \cup C'$ with a gainy cycle $C$, a lossy cycle $C'$ and a (possibly
trivial) oriented path $W$ from $C$ to $C'$; moreover,
the intersection of $C$ and $C'$ is an interval of both and $W$
is minimal as to make the bicycle connected. 
In addition we require that $C$ and $C'$ are equally
oriented on common arcs. See Fig.~\ref{fig:bicycles2} for two generic
examples.

   \calc_figscale{55}
    \begin{figure}[htb]
    \centerline{\input{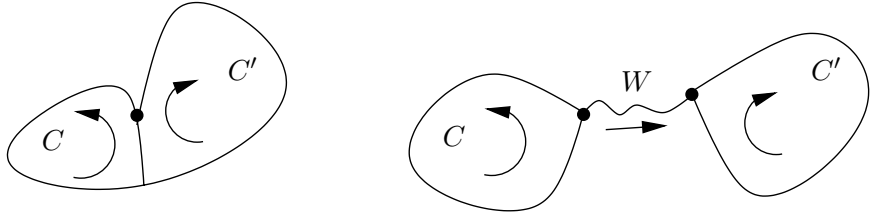}}
    \caption{Bicycles with $W=\emptyset$ and $W\not=\emptyset$.\label{fig:bicycles2}}
    \end{figure}
    

\begin{lem}\label{lem:incomp}
 A bicycle does not contain a breakeven cycle.
\end{lem}
\begin{proof}
  The cycles $C$ and $C'$ of a bicycle $H = C\cup W\cup C'$ are not breakeven.
  If $H$ contains an additional cycle $\widetilde{C}$, then the
  support of $\widetilde{C}$ must equal the symmetric difference of
  supports of $C$ and $C'$. Let $x:=\lambda(C\backslash C')$,
  $y:=\lambda(C\cap C')$, and $z:=\lambda(C'\backslash C)$, where
  orientations are taken according to $C$ and $C'$, respectively. We
  have $xy=\lambda(C)>1>\lambda(C')=zy$. Hence
  $\lambda(\widetilde{C})=(zx^{-1})^{\pm1}$, but
  $zx^{-1}=zy(xy)^{-1}<1$. That is, $\widetilde{C}$ cannot be
  breakeven.
\end{proof}

We call the set of bicycles and breakeven cycles of $D_{\Lambda}$ the
\nct{combinatorial support} for the set ${\mathcal{C}}(D_{\Lambda})$
of generalized cycles and denote it by $\underline{\mathcal{C}}(D_{\Lambda})$.
For $x\in\mathbb{R}^{m}$, let $\sign{x}$ be the \term{signed support} of $x$
, i.e., $\sign{x}$ is a partition $\supp{x}$ into positively and negatively oriented
 elements, where $i=\pm 1$ if $0\lessgtr x_i$, respectively.

Note that $\sign{\mathcal{C}(D_{\Lambda})}$ is exactly the
set of signed circuits of the \textit{oriented matroid} induced by the
matrix $N_{\Lambda}$, see~\cite{Bjo-93}. We justify the name
\textit{combinatorial support} by proving
$\underline{\mathcal{C}}(D_{\Lambda})=\sign{\mathcal{C}(D_{\Lambda})}$
in Theorem~\ref{thm:csII}. It turns out that oriented 
matroids arising as the combinatorial support of an arc-parameterized digraph 
are oriented versions of a combination of a classical \textit{cycle matroid} and 
a \textit{bicircular matroid}. The latter were introduced in the 
seventies~\cite{Mat-77, Sim-72}. Active research in the field can be 
found in~\cite{Gim-05,Gim-06,McN-08}. We feel that oriented matroids of
generalized network matrices are worth further investigation.

Given a walk $W=(a(0),\ldots, a(k))$ in $D$ we abuse notation and
identify $W$ with its \textit{signed support} $\sign{W}$, which is 
defined as the signed support of the signed incidence vector of $W$,
i.e, $\sign{W}:=\sign{\signvec{W}}$. Even more, 
we write $W_i$ and $W_{a(i)}$ for the same sign, namely the orientation 
of the arc $a(i)$ in $W$. Note that cycles and bicycles can be regarded 
to be walks; these will turn out to be the most interesting cases 
in our context.

A vector $f\subseteq\mathbb{R}^m$ is an \term{inner flow} of $W$ if
$\sign{f}=\pm\sign{W}$ and $f$ satisfies the generalized flow conservation law
between consecutive arcs of $W$.

\begin{lem}\label{lem:walk}
  Let $W=(a(0),\ldots, a(k))$ be a walk in $D_{\Lambda}$ and $f$ an inner flow
  of $W$. Then
$$
f_{a(k)}=K\lambda(W)^{-1}f_{a(0)}
$$
where the `correction term' $K$ is given by $K =
W_0W_k\lambda_{a(0)}^{\max(0,W_0)}\lambda_{a(k)}^{\min(0,W_k)}$.  In
particular the space of inner flows of $W$ is one-dimensional.
\end{lem}
\begin{proof}
 We proceed by induction on $k$.
 If $k=0$ then
\begin{eqnarray*}
  &~&{W}_{0}{W}_{0}\lambda_{a(0)}^{\max(0,{W}_{0})}\lambda_{a(0)}^{\min(0,{W}_{0})}\lambda(W)^{-1}f_{a(0)}\\
  &=&\lambda_{a(0)}^{{W}_{0}}\lambda(W)^{-1}f_{a(0)}\\
  &=&\lambda_{a(0)}^{{W}_{0}}\lambda_{a(0)}^{-{W}_{0}}f_{a(0)}\\
  &=& f_{a(0)}.
\end{eqnarray*}
If $k>0$ we can look at two overlapping walks $W'=(a(0),\ldots, a(\ell))$ and
$W''=(a(\ell),\ldots, a(k))$. Clearly $f$ restricted to $W'$ and $W''$
respectively satisfies the preconditions for the induction hypothesis. By applying
the induction hypothesis to $W''$ and $W'$ we obtain
\begin{eqnarray*}
f_{a(k)}&=& 
        {W}_{\ell}{W}_{k}\lambda_{a(\ell)}^{\max(0,{W}_{\ell})}
        \lambda_{a(k)}^{\min(0,{W}_{k})}\lambda(W'')^{-1}f_{a(\ell)}\quad\text{and}\\
f_{a(\ell)}&=&
        {W}_{0}{W}_{\ell}\lambda_{a(0)}^{\max(0,{W}_{0})}
        \lambda_{a(\ell)}^{\min(0,{W}_{\ell})}\lambda(W')^{-1}f_{a(\ell)}.
\end{eqnarray*}
Substitute the second formula into the first and observe that
${W}_{\ell}{W}_{\ell}=1$, and that from the product of four terms $\lambda_{a(\ell)}$
with different exponents the single $\lambda_{a(\ell)}^{-W_{\ell}}$ needed
for $\lambda(W'')^{-1}$ remains. This proves the claimed formula for $f_{a(k)}$.
\end{proof}

\begin{lem}\label{lem:path}
  Let $W=(a(0),\ldots, a(k))$ be a simple path from $v$ to $v'$ in
  $D_{\Lambda}$.
   If  $f$ is an
  inner flow of $W$ with $\sign{f}=\sign{W}$, then ${\omega(v,f)}<0$ and
  ${\omega(v',f)}>0$.
\end{lem}
\begin{proof}
  By definition $\omega(v,f)=-W_{0}\lambda_{a(0)}^{\max(0,W_{0})}f_{a(0)}$.
  Since $\lambda_{a(0)}>0$ and $\sign{f_{a(0)}}=W_{0}$ we conclude
  ${\omega(v,f)}<0$.  For the second inequality we use
  Lemma~\ref{lem:walk}:
\begin{eqnarray*}
	    \omega(v',f)&=&W_{k}\lambda_{a(k)}^{-\min(0,W_{k})}f_{a(k)}\\
		&=&W_{k}\lambda_{a(k)}^{-\min(0,W_{k})}W_{0}W_{k}
                       \lambda_{a(0)}^{\max(0,W_{0})}\lambda_{a(k)}^{\min(0,W_{k})}
                       \lambda(W)^{-1}f_{a(0)}\\
                &=&W_{0}\lambda_{a(0)}^{\max(0,W_{0})}\lambda(W)^{-1}f_{a(0)}.
\end{eqnarray*}
Since $\lambda_{a(0)},\lambda(W)^{-1}>0$ and
$\sign{f_{a(0)}}=\sign{W_{a(0)}}=W_{0}$ we
conclude ${\omega(v',f)}>0$.
\end{proof}

\begin{lem}\label{lem:cyc}
  Let $C=(a(0),\ldots, a(k))$ be a cycle in $D_{\Lambda}$ and $f$ an inner
  flow of $C$ with $\sign{f}=\sign{C}$. Then the excess $\omega(v,f)$ at the initial
  vertex $v$ satisfies $\sign{\omega(v,f)}=\sign{1-\lambda(C)}$.
\end{lem}
\begin{proof}
Reusing the computations from Lemma~\ref{lem:walk} we obtain
\begin{eqnarray*}
     \omega(v,f)&=&C_{k}\lambda_{a(k)}^{-\min(0,C_{k})}f_{a(k)}
                  -C_{0}\lambda_{a(0)}^{\max(0,C_{0})}f_{a(0)}\\
		&=&C_{0}\lambda_{a(0)}^{\max(0,C_{0})}\lambda(C)^{-1}f_{a(0)}
                  -C_{0}\lambda_{a(0)}^{\max(0,C_{0})}f_{a(0)}\\
                &=&C_{0}\lambda_{a(0)}^{\max(0,C_{0})}f_{a(0)}(\lambda(C)^{-1}-1).
\end{eqnarray*}
Since $\lambda_{a(0)}>0$ and $\sign{f_{a(0)}}=C_{0}$ 
we conclude $\sign{\omega(v,f)}=\sign{\lambda(C)^{-1}-1}$. 
Finally observe that $\sign{\lambda(C)^{-1}-1}=\sign{1-\lambda(C)}$.
\end{proof}

\begin{theorem}\label{thm:csI}
  Given a bicycle or breakeven cycle $H$ of $D_{\Lambda}$, the set of flows $f$
  with $\sign{f}=\pm \sign{H}$ is a 1-dimensional subspace of $\mathcal{F}(D_{\Lambda})$.
\end{theorem}
\begin{proof}
  Given $H\in\underline{\mathcal{C}}(D_{\Lambda})$ we want to characterize
  those $f\in\mathcal{F}(D_{\Lambda})$ with $\sign{f}=\pm \sign{H}$. Lemma~\ref{lem:walk}
  implies that the dimension of the inner flows of $H$ is at most one.
  Hence, it is enough to identify a single nontrivial flow on $H$.

  If $H=C\in\underline{\mathcal{C}}(D_{\Lambda})$ is a breakeven
  cycle, which traverses the arcs $(a(0),\ldots, a(k))$ starting and
  ending at vertex $v$.  By Lemma~\ref{lem:cyc} we have
  $\sign{\omega(v,f)}=\sign{1-\lambda(C)}$. Since $C$ is breakeven
  $\lambda(C)=1$, this implies generalized flow-conservation in $v$.
  Since by definition generalized flow-conservation holds for all
  other vertices we may conclude that $f$ is a generalized flow, i.e.,
  a nontrivial flow on $H$.

  Let $H\in\underline{\mathcal{C}}(D_{\Lambda})$ be a bicycle which
  traverses the arcs $(a(0),\ldots, a(k))$ such that $C=(a(0),\ldots, a(i)),
  W=(a(i+1),\ldots, a(j-1))$ and $C'=(a(j),\ldots, a(k))$. Let $v$ and $v'$ be
  the common vertices of $C$ and $W$ and $W$ and $C'$, respectively.

  Consider the case where $W$ is non-trivial. We construct
  $f\in\mathcal{F}(D_{\Lambda})$ with $\sign{f}=H$. First take any
  inner flow $f_C$ of $C$ with $\sign{f_C}=\sign{C}$. Since $C$ is
  gainy Lemma~\ref{lem:cyc} implies a positive excess at $v$. Let
  $f_W$ be an inner flow of $W$ with $\sign{f_W}=\sign{W}$.
  Lemma~\ref{lem:path} ensures $\omega(v,f_W)<0$. By scaling $f_W$
  with a positive scalar we can achieve $\omega(v,f_C+f_W)=0$. From
  Lemma~\ref{lem:path} we know that $f_C+f_W$ has positive excess at
  $v'$. Since $C'$ is lossy any inner flow $f_{C'}$ of $C'$ has
  negative excess at $v'$ (Lemma~\ref{lem:cyc}). Hence we can scale
  $f_{C'}$ to achieve $\omega(v',f_{C'}+f_W)=0$.  Together we have
  obtained  a generalized flow $f:=f_C+f_W+f_{C'}$, i.e.,
  a nontrivial flow on $H$.

  If $W$ is empty $v$ and $v'$ coincide. As in the above construction
  we can scale flows on $C$ and $C'$ such that $\omega(v,f) = 0$ holds
  for $f:=f_C+f_{C'}$, i.e., $f$ is a generalized flow.  If $C$ and $C'$ 
  share an interval the sign vectors of $C$
  and $C'$ coincide on this interval. From $\sign{f_C}=\sign{C}$
  and $\sign{f_C'}=\sign{C'}$ it follows that 
  $\sign{f}=\sign{C\cup C'}=\sign{H}$.  
  Hence $f$ is a flow on $H$.
\end{proof}

\begin{theorem}\label{thm:csII}
  For an arc parameterized digraph $D_{\Lambda}$ the set of supports
  of generalized cycles, i.e., of support minimal flows, coincides
  with the set of bicycles and breakeven cycles. Stated more formally:
  $\sign{\mathcal{C}(D_{\Lambda})}=\underline{\mathcal{C}}(D_{\Lambda})$.
\end{theorem}
\begin{proof}
  By Theorem~\ref{thm:csI} every $H\in\underline{\mathcal{C}}(D_{\Lambda})$
  admits a generalized flow $f$. To see support-minimality of $f$, assume that
  $H\in\underline{\mathcal{C}}(D_{\Lambda})$ has a strict subset $S$ which is
  support-minimal admitting a generalized flow. Clearly $S$ cannot have
  vertices of degree $1$ to admit a flow and must be connected to be
  support-minimal. Since $S \subset H\in\underline{\mathcal{C}}(D_{\Lambda})$ 
  this implies that is a cycle. Lemma~\ref{lem:cyc}
  ensures that $S$ must be a breakeven cycle. If
  $H$ was a breakeven cycle itself it cannot strictly contain $S$. Otherwise 
  if $H= C\cup W\cup C'$ is a bicycle then by Lemma~\ref{lem:incomp} it 
  contains no breakeven cycle.

  For the converse consider any $S\in\sign{\mathcal{C}(D_{\Lambda})}$,
  i.e., the signed support of some flow $f$.  We claim that 
  $\underline{S}:=\supp{f}$ contains a
  breakeven cycle or a bicycle. If it contains a breakeven cycle we
  are done, so we assume that it does not. Under this assumption it
  follows that there are two cycles $C_1,C_2$ in a connected component
  of $\underline{S}$.  If $C_1$ and $C_2$ intersect in at most one vertex, then we
  can choose the orientations for these cycles such that $\lambda(C_1)
  > 1$ and $\lambda(C_2) < 1$. If $C_1\cap C_2=\emptyset$ let $W$ be
  an oriented path from $C_1$ to $C_2$. Now $C_1\cup W \cup C_2$ is a
  bicycle contained in $\underline{S}$. The final case is that $C_1$ and $C_2$
  share several vertices. Let $B$ be a bow of $C_2$ over $C_1$, i.e, a
  consecutive piece of $C_2$ that intersects $C_1$ in its two
  endpoints $v$ and $w$ only. The union of $C_1$ and $B$ is a
  theta-graph, i.e, it consists of three disjoint path $B_1,B_2,B_3$ joining $v$
  and $w$, see Fig.~\ref{fig:theta}. Let the path $B_i$ be
oriented as shown in
  the figure and let $C = B_1 \cup B_2$ and $C' = B_2 \cup B_3$. If $C\cup C'$ is
  not a bicycle then the cycles are either both gainy or both lossy. 
  Assume that they are both gainy, i.e., $\lambda(C)  > 1$ and
  $\lambda(C')  > 1$. Consider the cycles $E = B_1 \cup B_3^{-1}$
  and $E' = B_1^{-1} \cup B_3$, since $\lambda(E) = \lambda(B_1)\lambda(B_3)^{-1} =
  \lambda(E')^{-1}$ it follows that either $E$ or $E'$ is a lossy cycle.
  The orientation of $E$ is consistent with $C$ and the orientation of $E'$ is consistent
  with $C'$. Hence either $C\cup E$ or $C'\cup E'$ is a bicycle contained in $\underline{S}$.
  This contradicts the support-minimality of $f$.
\end{proof}

   \calc_figscale{42}
    \begin{figure}[htb]
    \centerline{\input{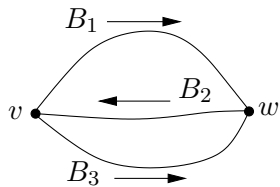}}
    \caption{A theta graph and an orientation of the three path.\label{fig:theta}}
    \end{figure}
    

For $H\in\underline{\mathcal{C}}(D_{\Lambda})$ we define $f(H)$ as the unique
$f\in\mathcal{C}(D)$ with $\sign{f(H)}=\sign{H}$ and $\Vert f(H)\Vert=1$. Let $x\in
\mathbb{R}^{m}$ and $H\in\underline{\mathcal{C}}(D_{\Lambda})$. Denote by
$\delta(H,x):=\langle x,f(H)\rangle$ the \term{bicircular balance} of
$x$ on $H$.

\begin{theorem}\label{thm:genflowdiff}
  Let $D_{\Lambda}$ be an arc-parameterized digraph and $x,c\in
  \mathbb{R}^{m}$. Then $x\in\mathcal{B}(D_{\Lambda})_{\leq c}$ if and only if
	\Item(1) \quad $x_a\leq c_a$ for all $a\in A$.
             \hfill \rm{(capacity constraints)}

	\Item(2) \quad $\delta(H,x)=0$ for all $H\in\underline{\mathcal{C}}(D_{\Lambda})$.
             \hfill \rm{(bicircular balance conditions)}
\end{theorem}

The theorem helps explain the name \textit{generalized bonds}:
usually a cocycle or bond is a set~$B$ of edges such that for every
cycle $C$ the incidence vectors are orthogonal, i.e., $\langle
x_B,x_C\rangle =0$. In our context the role of cycles is played by
generalized cycles, i.e., by generalized flows $f$ with
$\sign{f}=\sign{H}$ for some $H\in\underline{\mathcal{C}}(D_{\Lambda})$.

To make the statement of the theorem more general let $D_{\Lambda}$ be
an arc-parameterized digraph with arc capacities $c\in \mathbb{R}^{m}$
and a number $\Delta_H$ for each
$H\in\underline{\mathcal{C}}(D_{\Lambda})$.  A map $x: A \to
\mathbb{R}$ is called a \term{generalized $\Delta$-bond} if

\Item(B$_1$) \quad $x(a)\leq c_{u}(a)$ for all $a\in A$.
             \hfill (capacity constraints)

\Item(B$_2$) \quad $\delta(H,x)=\Delta_H$ for all $H\in\underline{\mathcal{C}}(D_{\Lambda})$.
             \hfill (bicircular $\Delta$-balance conditions)

\medskip

\ni Denote by $\mathcal{B}_{\Delta}(D_{\Lambda})_{\leq c}$ the set of
generalized $\Delta$-bonds of $D_{\Lambda}$. An argument as in the proof
of Lemma~\ref{lem:trans} yields the real-valued generalization of
Theorem~\ref{thm:bond} as a corollary of Theorem~\ref{thm:genbond}.

\begin{theorem}\label{thm:gendelta}
  Let $D_{\Lambda}$ be an arc-parameterized digraph with capacities
  $c\in \mathbb{R}^{m}$ and
  $\Delta\in\mathbb{R}^{\underline{\mathcal{C}}(D_{\Lambda})}$. The
  set $\mathcal{B}_{\Delta}(D_{\Lambda})_{\leq c}$ of generalized
  $\Delta$-bonds carries the structure of a distributive lattice and
  forms a polyhedron.
\end{theorem}

\section{Planar Generalized Flow}\label{sec:plan}

The \term{planar dual} $D^{*}$ of a non-crossing embedding of a planar
digraph $D$ in the sphere is an orientation of the planar dual $G^{*}$ of the
underlying graph $G$ of $D$: Orient an edge $\{v,w\}$ of $G^{*}$ from $v$ to
$w$ if it appears as a backward arc in the clockwise facial cycle of $D$ dual
to $v$. Call an arc-parameterized digraph $D_{\Lambda}$ \term{breakeven} if all its
cycles are breakeven.

\def\LambdaSt{\Lambda\kern-1pt^*}
\def\lambdaSt{\lambda\kern-1pt^*}

\begin{theorem}\label{thm:plan}
  Let $D_{\Lambda}$ be a planar breakeven digraph. There is an arc
  parameterization ${\LambdaSt}$ of the dual $D^*$ of $D$ such that
  $\mathcal{F}(D_{\Lambda}) \cong \mathcal{B}(D^*_{\LambdaSt})$. More
  precisely, there is a vector $\sigma \in \mathbb{R}^m$ with positive
  components such that $f$ is a generalized flow of $D_{\Lambda}$ if and only if
  $x = S(\sigma) f$ is a generalized bond of $D^*_{\LambdaSt}$, where
  $S(\sigma)$ denotes the diagonal matrix with entries from $\sigma$.
\end{theorem}
\begin{proof}
  Let $C_1\ldots C_{n^*}$ be the list of clockwise oriented facial
  cycles of $D$.  For each $C_i$ let~$f_i$ be a generalized flow with
  $\sign{f_i} = \sign{C_i}$; since $C_i$ is breakeven such an $f_i$ exists by
  Lemma~\ref{lem:cyc}.  Collect the flows $f_i$
  as rows of a matrix $M$. Columns of $M$ correspond to edges of~$D$
  and due to our selection of cycles each column contains exactly two
  nonzero entries. The orientation of the facial cycles and the sign
  condition implies that each column has a positive and a negative
  entry. For the column of arc $a$ let $\mu_a > 0$ and $\nu_a < 0$ be
  the positive and negative entry. Define $\sigma_a := \mu_a^{-1} > 0$
  and note that scaling the column of $a$ with $\sigma_a$ yields entries 
  $1$ and $-\lambdaSt_a = \nu_a\mu_a^{-1} < 0$ in this column.
  Therefore, $N_{\LambdaSt} := M S(\sigma)$ is a generalized network matrix. 
  The construction implies that the underlying digraph of $N_{\LambdaSt}$
  is just the dual $D^*$ of $D$.

  Let $f\in\mathcal{F}(D_{\Lambda})$ be a flow. Then $f$ can be expressed as
  linear combination of generalized cycles. Since $D_{\Lambda}$ is
  breakeven we know that the support of every generalized cycle is a
  simple cycle. The facial cycles generate the cycle space of $D$.
  Moreover, if $C$ is a simple cycle and $f_C$ is a flow with
  $\sign{f_C} = \sign{C}$, then $f_C$ can be expressed as a linear
  combination of the flows $f_i$, $i=1,\ldots,n^*$ (exercise).
  This implies that the rows of $M$ are spanning for $\mathcal{F}(D_{\Lambda})$, i.e., 
  for every $f$ there is a $q\in\mathbb{R}^{n^*}$ such that $f = M^\top q$.
  In other words $\mathcal{F}(D_{\Lambda}) = M^\top \mathbb{R}^{n^*}$.
  
  A vector $x$ is a bond for $N_{\LambdaSt}$ if and only if
  $x$ is in the row space of $N_{\LambdaSt}$,
  i.e., there is a potential $p\in\mathbb{R}^{n^*}$ with $x = N_{\LambdaSt}^\top p$.
  In other words $\mathcal{B}(D^*_{\LambdaSt})=N_{\LambdaSt}^\top\mathbb{R}^{n^*}
  = (M S(\sigma))^\top\mathbb{R}^{n^*} =  S(\sigma) M^\top\mathbb{R}^{n^*}
  = S(\sigma)\mathcal{F}(D_{\Lambda})$.
\end{proof}

\begin{corollary}\label{cor:plan}
  Let $D_{\Lambda}$ be a planar breakeven digraph and $c\in\mathbb{R}^{m}$.
  The set $\mathcal{F}(D_{\Lambda})_{\leq c}$ carries the structure of a
  distributive lattice.
\end{corollary}
\begin{proof}
  The matrix $S(\sigma)$ is an isomorphism between
  $\mathcal{F}(D_{\Lambda})$ and $\mathcal{B}(D^*_{\LambdaSt})$. Since
  $\sigma$ is positive we obtain $\mathcal{F}(D_{\Lambda})_{\leq c}
  = S(\sigma)\big(\mathcal{B}(D^{*}_{\LambdaSt})_{\leq
    S(\sigma)c}\big)$. Theorem~\ref{thm:gendelta} implies a distributive
   lattice structure on $\mathcal{B}(D^{*}_{\LambdaSt})_{\leq
    S(\sigma)c}$ which can be pushed to
   $\mathcal{F}(D_{\Lambda})$.
\end{proof}

In fact Theorem~\ref{thm:gendelta} even allows us to obtain a
distributive lattice structure for planar generalized flows of
breakeven digraphs with an arbitrarily prescribed excess at every
vertex.
\medskip

The reader may worry about the existence of non-trivial
arc-parameterizations $\Lambda$ of a digraph $D$ such that $D_\Lambda$
is breakeven. Here is a nice construction for such parameterizations.
Let $D$ be arbitrary and $x\in\mathbb{R}^m$ be a $\mathbf{0}$-bond of $D$, i.e.,
$\delta(C,x):=\sum_{a\in C^{+}}x_a-\sum_{a\in C^{-}}x_a = 0$ for all
oriented cycles $C$. Consider $\lambda=\exp(x)$ and note that $\lambda_a \geq 0$
for all arcs $a$ and that
$\lambda(C)=
\big(\prod_{a\in C^{+}}\lambda_{a}\big)\big(\prod_{a\in C^{-}}\lambda_{a}\big)^{-1}
=\exp( \delta(C,x) ) = 1$
for all oriented cycles $C$. Hence weighting the arcs of with $\lambda$
yields a breakeven arc-parameterization of $D$.
This construction is universal in the sense that application of the logarithm 
to a breakeven parameterization yields a $\mathbf{0}$-bond.

\section{Conclusions and Open Questions}\label{sec:conco}
\def\Problem{\par\smallskip\ni{\bf Problem.}\ }
\def\topic#1{\par\ni{\sf #1}\ }

\topic{Old and New:} In the present paper we have obtained a distributive
lattice representation for the set of real-valued generalized
$\Delta$-bonds of an arc parameterized digraph. The proof is based on
the bijection with potentials which allows us to push the obvious lattice
structure based on componentwise max and min from potentials to
generalized bonds. Consequently we obtain a distributive lattice on
generalized bonds in terms of its join and meet.  In~\cite{Fel-08} we
obtained the distributive lattice structure on integral
$\Delta$-bonds, by showing that we can build the cover-graph of a
distributive lattice by local vertex-push-operations and reach every
$\Delta$-bond this way.  This qualitatively different distributive
lattice representation was possible because we could assume the
digraph to be \textit{reduced} in a certain way. 
\Problem Is there a way to
reduce an arc-parameterized digraph such that the distributive lattice
on its generalized bonds can be constructed locally by
\textit{pushing} vertices?

\bigskip

\topic{Order Theory:} There is a natural finite distributive lattice
associated to a D-polyhedron $P$. Start from the vertices of $P$ and
consider the closure of this set under join and meet.  Let $L(P)$ be
the resulting distributive \textit{vertex lattice} of $P$. It would be
interesting to know what information regarding $P$ is already contained
in $L(P)$.  
\Problem What do the generalized bonds associated to the
elements of $L(P)$ look like? In particular some special generalized bonds
of $L(P)$ including join-irreducible, minimal and maximal elements are of interest.

\bigskip

\topic{Geometry:} We have derived an $\mathcal{H}$-description of
D-polyhedra.
\Problem What does a $\mathcal{V}$-description
look like?\\
(This again asks for a special set of elements
of the vertex lattice $L(P)$.)

In fact, the previous problem can be `turned around': For every
distributive lattice $L$ there are integral D-polyhedra such that the 
integral points in the polyhedron form a lattice isomorphic to $L$.
\Problem Which subsets of $L$ can arise as sets of vertices of such
polyhedra? 
\bigskip

\topic{Optimization:} There has been a considerable amount of research
concerned with algorithms for generalized flows, see~\cite{Ahu-93}
for references. As far as we know it has never been taken into
account that the LP-dual problem of a min-cost generalized flow is an
optimization problem on a D-polyhedron. We feel that it might be fruitful to
look at this connection. A special case is given by generalized
flows of planar breakeven digraphs, where the flow-polyhedron
also forms a distributive lattice (Corollary~\ref{cor:plan}).

In particular, it would be interesting to understand the integral
points of a D-polyhedron, which by Observation~\ref{obs:inter} form a
distributive lattice. Related to this and to~\cite{Fel-08} is the
following:
\Problem
Find conditions on $\Lambda$ and $c$ such that the set of
integral generalized bonds for these parameters forms 
a distributive lattice.

\subsection*{Acknowledgments}
We thank G\"unter Rote for the hint to look at generalized flows,
Torsten Ueckerdt for fruitful discussions and Anton Dochterman for his
help with the exposition.


 \bibliography{Bibstuff/literature,Bibstuff/antrag07-bib}
\bibliographystyle{Bibstuff/my-siam}


\end{document}